\documentclass{amsart}

\usepackage[utf8]{inputenc}
\usepackage{color}
\usepackage{xcolor}
\usepackage{hyperref}
\usepackage{amssymb}
\usepackage{float}
\usepackage{tikz}
\usepackage{forest}
\usepackage{nicefrac}
\usepackage{enumerate}
\usepackage[nameinlink, capitalize, noabbrev]{cleveref}
\usetikzlibrary{shapes.geometric,decorations.markings}
\usetikzlibrary{decorations.pathreplacing}
\usetikzlibrary{fit}
\usetikzlibrary{automata}
\usetikzlibrary{positioning}
\usetikzlibrary{intersections}
\tikzset{mytext/.style={font=\small, text=black}}

\hypersetup{
    colorlinks=true,
    linkcolor=teal,
    citecolor=magenta,
    }

\newtheorem{Theorem}{Theorem}

\newtheorem{proposition}{Proposition}[section]
\newtheorem{lemma}[proposition]{Lemma}
\newtheorem{corollary}[proposition]{Corollary}
\newtheorem{theorem}[proposition]{Theorem}

\theoremstyle{definition}

\newtheorem{definition}[proposition]{Definition}

\numberwithin{equation}{section}

\title[Markov processes and fractal groups]{Markov processes associated to fractal branch groups}
\author{Jorge Fariña-Asategui}
\address{Jorge Fariña-Asategui: Centre for Mathematical Sciences, Lund University, 223 62 Lund, Sweden -- Department of Mathematics, University of the Basque Country UPV/EHU, 48080 Bilbao, Spain}
\email{jorge.farina\_asategui@math.lu.se}
\thanks{The author is supported by the Spanish Government, grant PID2020-117281GB-I00, partly with FEDER funds. The author also acknowledges support from the Walter Gyllenberg Foundation from the Royal Physiographic Society of Lund}

\begin{document}

\begin{abstract}

The author introduced recently a new natural construction which associates a measure-preserving dynamical system to any fractal profinite group. Here, we investigate these measure-preserving dynamical systems under the extra assumption on the groups to be branch. First, we compute their $f$-invariant, a measure-conjugacy invariant introduced by Bowen, and show that they are Markov processes over free semigroups in the sense of Bowen. Secondly, we show that fractal branch profinite groups with the same Hausdorff dimension and whose associated measure-preserving dynamical systems have the same $f$-invariant yield isomorphic Markov processes.
\end{abstract}

\maketitle

\section{introduction}
\label{section: introduction}

The simplest example of a measure-preserving dynamical system is the Bernoulli shift, i.e. the space of two-sided infinite sequences on a finite alphabet endowed with an invertible shift operator and a shift-invariant probability measure. The main tool in the study of the classical Bernoulli shift is entropy, first introduced by Kolmogorov \cite{Kolmogorov1,Kolmogorov2} and later modified by Sinai \cite{Sinai}. The celebrated result of Ornstein in the groundbreaking papers \cite{Ornstein1,Ornstein2} showed that entropy is indeed a complete measure-conjugacy invariant for two-sided Bernoulli shifts. Entropy theory has been successfully extended to amenable group actions \cite{Kieffer, OrnsteinWeiss}. 

However, for non-amenable group actions, entropy theory is harder. The prototype of a non-amenable group is the free group, so a first step in understanding non-amenable group actions is to understand free group actions. In his remarkable work in \cite{BowenF}, Bowen introduced the $f$-invariant for measure-preserving free group actions, a non-amenable analogue of Kolmogorov-Sinai entropy, and showed that it is a complete measure-conjugacy invariant for Bernoulli shifts over a free group. 

One of the simplest dynamical systems after the Benoulli shifts are Markov processes. In the classical setting, i.e. for $\mathbb{Z}$-actions, Ornstein Isomorphism Theorem still applies to Markov processes \cite{FriedmanOrnstein}. In the non-amenable setting, Bowen introduced Markov processes over free groups and semigroups in~\cite{BowenMarkov}. In general, the semigroup actions are far less understood; there is no nice entropy theory even in the classical setting of $\mathbb{N}$-actions.

In this paper, we consider free semigroup actions. In fact, we shall consider a large family of measure-preserving dynamical systems, introduced by the author in \cite{JorgeCyclicity}, arising from fractal profinite groups acting on regular rooted trees; see \cref{section: Preliminaries} for the unexplained terms here and elsewhere in the introduction. Given a profinite fractal group $G\le \mathrm{Aut}~T$, we consider the probability space $(G,\mu_G)$, where $\mu_G$ denotes the Haar measure of $G$. We further consider the free monoid action $\mathcal{T}$ (where we identify the regular rooted tree $T$ with a free monoid) given by taking sections, i.e. $\mathcal{T}_v(g):=g|_v$ for all $v\in T$ and all $g\in G$. Then, it was proven in \cite[Theorem A]{JorgeCyclicity} that the action $\mathcal{T}$ is measure-preserving, and one obtains a measure-preserving dynamical system $(G,\mu_G,T,\mathcal{T})$. These measure-preserving dynamical systems have already found further applications to other areas in mathematics, such as to arithmetic dynamics and number theory; see \cite{JorgeSanti}. 

Our first result in this paper concerns the computation of the $f$-invariant of these measure-preserving dynamical systems. Furthermore, we obtain that these measure-preserving dynamical systems yield Markov processes over free semigroups (see \cref{definition: Markov}) when the group under consideration is further assumed to be branch. The class of branch groups was introduced by Grigorchuk in 1997 and it includes examples of Burnside groups, groups of intermediate growth and amenable but not elementary amenable groups; see \cite{Handbook, SelfSimilar} for an overview on these groups.

For a group $G\le \mathrm{Aut}~T$, we recall the definition of the sequence $\{r_n(G)\}_{n\ge 1}$ from~\cite{JorgeSpectra}. For any $n\ge 1$, we define $r_n(G)$ as 
$$r_n(G):=m\log|G_n|-\log|G_{n+1}|+\log|G_1|.$$
The sequence $\{r_n(G)\}_{n\ge 1}$ was introduced by the author in \cite[Section 3]{JorgeSpectra} in order to compute the Hausdorff dimension of self-similar profinite groups; see also \cite[Section 4.2]{GeneralizedBasilica} for the related series of obstructions. 

Remarkably, the sequence $\{r_n(G)\}_{n\ge 1}$ essentially gives the $f$-invariant of the measure-preserving dynamical system associated to a fractal branch profinite group:

\begin{Theorem}
\label{Theorem: f-invariant of fractal of finite type}
    Let $G\le \mathrm{Aut}~T$ be a fractal branch profinite group. Then, there exists some $D\ge 1$ such that the $f$-invariant of the measure-preserving dynamical system $(G,\mu_G,T,\mathcal{T})$ is given by
    $$f(G)=F(T,\alpha_s^{D})=\log|G_1|-r_{D}(G).$$
    In particular, the process $(G,\mu_G,T,\mathcal{T},\alpha_s^D)$, where $\alpha_s$ denotes the standard partition of $(G,\mu_G)$, is Markov.
\end{Theorem}

In the context of fractal profinite groups it was shown by the author in \cite[Theorem 3.7]{JorgeSpectra} that branch groups are in fact regular branch and thus, by a well-known result of Grigorchuk and \v{S}uni\'{c}, they are groups of finite type; see \cref{section: Preliminaries}. We shall see that the $D$ appearing in the statement of \cref{Theorem: f-invariant of fractal of finite type} is no more than the depth of $G$ as a group of finite type.

\cref{Theorem: f-invariant of fractal of finite type} provides a large family of examples of Markov processes over free semigroups: the first and the second Grigorchuk group, the non-constant Grigorchuk-Gupta-Sidki groups (GGS-groups for short), and the Hanoi Towers group and its generalizations, among others.

Even in the classical setting, i.e. for $\mathbb{N}$-actions, classifying Markov processes is more complicated than the corresponding problem over the group $\mathbb{Z}$. Our second aim in this paper is to give sufficient conditions for two Markov processes arising from fractal branch profinite groups to be isomorphic. It turns out that this can be done in terms of the Hausdorff dimension:

\begin{Theorem}
\label{Theorem: classification}
    Let $G,H\le \mathrm{Aut}~T$ be two fractal branch profinite groups such that:
    \begin{enumerate}[\normalfont(i)]
        \item $f(G)=f(H)$;
        \item $\mathrm{hdim}_{\mathrm{Aut}~T}(G)=\mathrm{hdim}_{\mathrm{Aut}~T}(H)$.
    \end{enumerate}
    
    Let $\alpha_s$ and $\beta_s$ denote the standard partitions of $(G,\mu_G)$ and $(H,\mu_H)$ respectively. Then, there exists some $D\ge 1$ such that the Markov processes $(G,\mu_G,T,\mathcal{T},\alpha_s^D)$ and $(H,\mu_H,T,\mathcal{T},\beta_s^D)$ are isomorphic.
\end{Theorem}

\cref{Theorem: classification} yields a new application of the Hausdorff dimension of self-similar profinite groups. In fact, in \cref{section: applications}, we apply \cref{Theorem: classification} to the family of non-constant GGS-groups acting on the $p$-adic tree. Together with previous results of Fernández-Alcober and Zugadi-Reizabal in \cite{GGSHausdorff}, we use \cref{Theorem: classification} to deduce that symmetric (respectively non-symmetric) defining vectors whose circulant matrix are of the same rank yield GGS-groups giving rise to isomorphic Markov processes; see the discussion after \cref{corollary: f invariant GGS}.

\subsection*{\textit{\textmd{Organization}}} In \cref{section: Preliminaries} we introduce some background material on groups acting on regular rooted trees and the associated measure-preserving dynamical systems. We further introduce the $f$-invariant and Markov processes. In \cref{section: the f-invariant}, groups of finite type are discussed and we compute the $f$-invariant of the associated measure-preserving dynamical systems, proving \cref{Theorem: f-invariant of fractal of finite type}. We further recall the notion of Hausdorff dimension in the context of self-similar profinite groups and prove \cref{Theorem: classification}. \cref{section: applications} is devoted to further applications of the main results in \cref{section: the f-invariant}. We conclude the paper by showing what \textcolor{teal}{Theorems} \ref{Theorem: f-invariant of fractal of finite type} and \ref{Theorem: classification} say in the case of the non-constant GGS-groups acting on $p$-adic trees.

\subsection*{\textit{\textmd{Notation}}} Groups will be assumed to act on the tree on the right so composition will be written from left to right. We shall use exponential notation for group actions on the tree. Finally, we denote by $\# S$ the cardinality of a finite set $S$.

% \subsection*{Acknowledgements} 

\section{Fractal groups and measure-preserving dynamical systems}
\label{section: Preliminaries}

In this section, we introduce the background on groups acting on regular rooted trees and on measure-preserving dynamical systems that will be needed in subsequent sections.

\subsection{Groups acting on regular rooted trees}

For a natural number $m\ge 2$ and a finite set of $m$ symbols $\{1,\dotsc,m\}$, we define the \textit{free monoid} on the set $\{1,\dotsc,m\}$ as the monoid consisting of finite words with letters in $\{1,\dotsc,m\}$. The free monoid can be identified with the \textit{$m$-adic tree}, i.e. the rooted tree $T$ where each vertex has exactly $m$ immediate descendants. The words in~$T$ of length exactly $n$ form the \textit{$n$th level of} $T$. We may also use the term level to refer to the number $n$.

Let $\mathrm{Aut~}T$ be the group of graph automorphisms of the $m$-adic tree $T$. It is easy to see that the automorphisms of $T$ fix the root of $T$ and act by permuting the vertices at the same level of $T$.

For any $1\le n\le \infty$, the \textit{$n$th truncated tree} $T_n$ consists of the vertices at distance at most $n$ from the root. Note that $T_\infty=T$. We denote the group of automorphisms of the $n$th truncated tree by $\mathrm{Aut~}T_n$. Let $g\in\mathrm{Aut~}T$ and $v\in T$. For $1\le n\le \infty$, we define the \textit{section of $g$ at $v$ of depth }$n$ as the unique automorphism $g|_v^n\in\mathrm{Aut~}T_n$ such that 
$$(vu)^g=v^gu^{g|_v^n}$$
for every $u\in T_n$. For $n=\infty$, we simply write  $g|_v$ and call it the \textit{section of $g$ at $v$}.

For every $n\ge 1$, the normal subgroup $\mathrm{St}(n)$ of finite index in $\mathrm{Aut}~T$  consisting of automorphisms fixing all the vertices of the $n$th level of $T$ is called the \textit{$n$th level stabilizer}. Similarly, for any vertex $v\in T$, we define the \textit{vertex stabilizer} $\mathrm{st}(v)$ as the subgroup of automorphisms fixing the vertex $v$. 

The group $\mathrm{Aut~}T$ is a countably based profinite group with respect to the topology induced by the level stabilizers. We call this topology the \textit{congruence topology}. As a profinite group, $\mathrm{Aut~}T$ is endowed with a unique normalized Haar measure. Furthermore, since closed subgroups of a profinite group are themselves profinite, any closed subgroup $G\le \mathrm{Aut}~T$ admits a unique normalized Haar measure, which we denote by $\mu_G$. We shall study certain measure-preserving transformations on the probability space $(G,\mu_G)$.

Let $G$ be a subgroup of $\mathrm{Aut}~T$ for the remainder of the section. We define $\mathrm{st}_G(v):=\mathrm{st}(v)\cap G$ and $\mathrm{St}_G(n):=\mathrm{St}(n)\cap G$ for any vertex $v$ and any level $n\ge 1$, respectively. The quotients $G_n:=G/\mathrm{St}_G(n)$ are called the \textit{congruence quotients} of~$G$.

The group $G$ is \textit{level-transitive} if it acts transitively on all the levels of $T$. We say that $G$ is \textit{self-similar} if for any $g\in G$ and any vertex $v\in T$ we have $g|_{v}\in G$. We shall say that $G$ is \textit{fractal} if $G$ is level-transitive, self-similar and $\mathrm{st}_G(v)|_v=G$ for every $v\in T$. A stronger version of fractality is that of \textit{strongly fractal} groups, where $\mathrm{St}_G(1)|_v=G$ for every vertex $v$ at the first level of $T$ for every level $n\ge 1$. However, by level-transitivity, it is enough to check the condition $\mathrm{St}_G(1)|_v=G$ on just a single vertex at the first level of $T$.

Let $\mathrm{rist}_G(v)\le G$ be the subgroup consisting of automorphisms fixing $v$ and every vertex which is not a descendant of $v$. The subgroup $\mathrm{rist}_G(v)$ is called the \textit{rigid vertex stabilizer} of $v$ in $G$. For distinct vertices at the same level of $T$, the corresponding rigid vertex stabilizers commute and the direct product of all the rigid vertex stabilizers at a level $n$ is called the \textit{rigid level stabilizer} of level $n$ in $G$ and it is denoted by $\mathrm{Rist}_G(n)$. Note that $\mathrm{Rist}_G(n)$ is a normal subgroup of $G$. If $G$ is level-transitive and for every $n\ge 1$ the rigid stabilizer $\mathrm{Rist}_G(n)$ is of finite-index in~$G$ we say that $G$ is \textit{branch}.

A stronger notion of branchness is defined as follows. A subgroup $K\le \mathrm{Aut}~T$ is called \textit{branching} if for every $v\in T$ we have $\mathrm{rist}_K(v)|_v\ge K$. A level-transitive group $G\le \mathrm{Aut}~T$ is said to be \textit{regular branch over $K$} if it contains a finite index branching subgroup $K$.

We conclude the introduction to groups acting on rooted trees by defining the sequence $\{r_n(G)\}_{n\ge 1}$. For any $n\ge 1$, we define $r_n(G)$ as 
$$r_n(G):=m\log|G_n|-\log|G_{n+1}|+\log|G_1|.$$
The sequence $\{r_n(G)\}_{n\ge 1}$ was introduced by the author in \cite{JorgeSpectra} for the study of the Hausdorff dimension of groups acting on regular rooted trees (note that the forward gradient of the sequence $\{r_n(G)\}_{n\ge 1}$ coincides with the series of obstructions introduced by Petschick and Rajeev in \cite{GeneralizedBasilica}). Here, we shall see that, remarkably, the sequence $\{r_n(G)\}_{n\ge 1}$ also arises naturally in the study of the $f$-invariant of Markov processes associated to fractal branch profinite groups.

\subsection{Measure-preserving dynamical systems}

Let $(\Omega,\mu)$ be a probability space and let $S$ be a monoid. We fix an action $\mathcal{S}$ of $S$ on $(\Omega,\mu)$ and we say this monoid action is \textit{measure-preserving} if for any measurable subset $Y\subseteq \Omega$ and any $s\in S$ we have $\mu(\mathcal{S}_s^{-1}(Y))=\mu(Y)$, where $\mathcal{S}_s$ is the operator associated to the action of $s$. Then the tuple $(\Omega,\mu,S,\mathcal{S})$ is called a \textit{measure-preserving dynamical system}. 

Given a finite measurable partition $\alpha$ of $(\Omega,\mu)$, we call the tuple $(\Omega,\mu,S,\mathcal{S},\alpha)$ an \textit{$S$-process}.

\begin{definition}[Isomorphism of processes]
\label{definition: measure-conjugacy}
    Two $S$-processes $(\Omega,\mu,S,\mathcal{S},\alpha)$ and $(\widetilde{\Omega},\nu,S,\widetilde{\mathcal{S}},\beta)$ are \textit{isomorphic} if there exist conull sets $\Omega'\subseteq \Omega$ and $\widetilde{\Omega}'\subseteq \widetilde{\Omega}$ and a measurable map $\phi:\Omega'\to \widetilde{\Omega}'$ with measurable inverse $\phi^{-1}:\widetilde{\Omega}'\to \Omega'$ such that:
    \begin{enumerate}[\normalfont(i)]
        \item $\phi$ is measure-preserving, i.e. $\mu(\phi^{-1}(A))=\nu(A)$ for any $\nu$-measurable subset $A\subseteq \widetilde{\Omega}'$;
        \item $\phi(\mathcal{S}_s(x))=\widetilde{\mathcal{S}}_s\phi(x)$ for all $s\in S$ and $x\in \Omega'$;
        \item $\phi$ induces a bijection between the partitions $\alpha$ and $\beta$.
    \end{enumerate}
\end{definition}

In \cite{JorgeCyclicity}, the author introduced a natural way to associate a measure-preserving dynamical system to a fractal profinite group. Let us consider $G\le \mathrm{Aut}~T$ a fractal closed subgroup and write $(G,\mu_G)$ for the probability space, where $\mu_G$ denotes the Haar measure in $G$. The standard fact that
$$\mu_G(g\mathrm{St}_G(n))=\mu_G(\mathrm{St}_G(n))=|G_n|^{-1}$$
for all $n\ge 1$ will be used throughout the paper.

We regard the $m$-adic tree $T$ as the free monoid of rank $m$ and define the monoid action $\mathcal{T}$ on $(G,\mu_G)$ via sections, i.e. $\mathcal{T}_v(g)=g|_v$ for every $v\in T$ and any $g\in G$. This monoid action $\mathcal{T}$ is measure-preserving:

\begin{theorem}[{{see \cite[Theorem A]{JorgeCyclicity}}}]
    \label{theorem: fractal is dyn sys}
    Let $G\le \mathrm{Aut}~T$ be a fractal profinite group. Then $(G,\mu_G,T,\mathcal{T})$ is a measure-preserving dynamical system.
\end{theorem}

\subsection{The $f$-invariant}
The $f$-invariant was introduced by Bowen in \cite{BowenF} as a measure-conjugacy invariant for free group measure-preserving actions; see \cite{BowenMarkov} for the analogous definition for the free semigroup case. Let us recall its definition in the free semigroup case.

Let $T$ be the free monoid (semigroup with identity $\emptyset$) on the set $\{1,\dotsc,m\}$ and let $(\Omega,\mu)$ be a probability space, where $T$ acts on $\Omega$ via a measure-preserving action~$\widetilde{\mathcal{T}}$. We write $\mathcal{P}$ for the set of all measurable finite partitions of $(\Omega,\mu)$. For a partition $\alpha\in \mathcal{P}$ and a finite subset $Q\subseteq T$ we write
$$\alpha^Q:=\bigvee_{q\in Q}\widetilde{\mathcal{T}}_q^{-1}\alpha,$$
where the join of two partitions $\alpha\vee \beta$ is the partition into the sets $A\cap B$ with $A\in \alpha$ and $B\in \beta$. In the special case when $Q=B_T(\emptyset,n)$, i.e. the ball of radius $n$ centered at the identity, we simply write $\alpha^n:=\alpha^{B_T(\emptyset,n)}$.

For a partition $\alpha\in \mathcal{P}$, its \textit{Shannon entropy} $H(\alpha)$ is defined as
$$H(\alpha):=-\sum_{A\in \alpha}\mu(A)\log(\mu(A)).$$
We further define the quantity $F(T,\alpha)$ as
$$F(T,\alpha):=(1-2m)H(\alpha)+\sum_{i=1}^m H(\alpha\vee \widetilde{\mathcal{T}}_{i}^{-1}(\alpha))$$
and $f(T,\alpha)$ as
$$f(T,\alpha):=\inf_{n\ge 1}F(T,\alpha^n).$$

Finally, we define the \textit{$f$-invariant} of the measure-preserving dynamical system $(\Omega,\mu,T,\widetilde{\mathcal{T}})$ as $f(\Omega):=f(T,\alpha)$ for any generating partition $\alpha$ (if such a partition exists).

\subsection{Markov processes}

We now define Markov processes over the free monoid $T$ following Bowen in \cite[Section 6]{BowenMarkov}. The definition is just the natural generalization of usual Markov chains, i.e. an stochastic process $\{X_n\}_{n\ge 1}$ such that the distribution of $X_n$ conditioned on $\sigma(\bigcup_{i=1}^{n-1}X_i)$ is the same as the distribution of $X_n$ conditioned on $\sigma(X_{n-1})$. 

To make this intuition precise, we need some definitions. We fix $T$ the free monoid of rank $m$ and $X:=\{1,\dotsc,m\}$ a generating set for $T$. We shall use the identification of $T$ with the $m$-adic tree.

\begin{definition}[Past of a vertex]
Given a vertex $v\in T$, we define its \textit{past} $\mathrm{Past}(v)$ as the set of vertices in the unique path from $v$ to the root (including both $v$ and the root).
\end{definition}

If we write a vertex $v=x_1\dotsb x_n\in T$ with each $x_i\in X$, then
$$\mathrm{Past}(v)=\{x_1\dotsc x_i\mid 1\le i\le n\}\cup\{\emptyset\}.$$

\begin{definition}[Markov process {\cite[Definition 20]{BowenMarkov}}]
\label{definition: Markov}
A $T$-process $(\Omega,\mu, T, \widetilde{\mathcal{T}},\alpha)$ is a \textit{Markov} process if for every $x\in X$, every $v\in T$ and every $A\in \alpha$ we have
$$\mu\Big(\mathcal{T}_{vx}^{-1}(A) \mathrel{\Big|} \bigvee_{w\in \mathrm{Past}(v)}\mathcal{T}_{w}^{-1}(\alpha)\Big)=\mu(\mathcal{T}_{vx}^{-1}(A)\mid \mathcal{T}_{v}^{-1}(\alpha))=\mu(\mathcal{T}_{x}^{-1}(A)\mid \alpha),$$
where we write $\mu(\cdot\mid \mathcal{F})$ for the conditional probability on a sub-$\sigma$ algebra $\mathcal{F}$. Note that the second equality always holds as $\widetilde{\mathcal{T}}$ is measure-preserving.
\end{definition}

The $f$-invariant characterizes Markov processes:

\begin{theorem}[{see {\cite[Theorem 11.1]{BowenMarkov}}}]
\label{theorem: Markov characterization}
    An $S$-process $(\Omega,\mu,T,\widetilde{\mathcal{T}}, \alpha)$ is Markov if and only if $f(\Omega)=F(T,\alpha)$.
\end{theorem}

\section{Fractal groups of finite type}
\label{section: the f-invariant}

In this section, we first introduce some preliminary results on groups of finite type. Next, we compute the $f$-invariant of the dynamical systems arising from fractal groups of finite type proving \cref{Theorem: f-invariant of fractal of finite type}. Finally, we recall the notion of Hausdorff dimension in the context of self-similar groups and prove \cref{Theorem: classification}.

\subsection{Groups of finite type}

A group $G\le \mathrm{Aut}~T$ is said to be of \textit{finite type} if there exists some $D\ge 1$ and some subgroup $H\le \mathrm{Sym}(m)\wr\overset{D}{\ldots}\wr \mathrm{Sym}(m)$ such that
$$G=\{g\in \mathrm{Aut}~T\mid g|_v^D\in H\text{ for every }v\in T\}.$$
In that case, the natural number $D$ is called the \textit{depth} of $G$ and the subgroup $H$ the set of \textit{defining patterns} of $G$. Note that for $D=1$, we simply obtain the iterated wreath products of a subgroup $H\le \mathrm{Sym}(m)$.

It is clear by definition that groups of finite type are closed subgroups of $\mathrm{Aut}~T$. Furthermore, the following result of Grigorchuk and \v{S}uni\'{c}, shows that level-transitive groups of finite type are precisely regular branch closed subgroups of $\mathrm{Aut}~T$:

\begin{theorem}[{{see \cite[Theorem 3]{SunicHausdorff} and \cite[Proposition 7.5]{GrigorchukFinite}}}]
    \label{theorem: equivalence finite type}
    Let $G\le \mathrm{Aut}~T$ be a level-transitive closed subgroup. Then the following are equivalent:
    \begin{enumerate}[\normalfont(i)]
        \item $G$ is of finite type of depth $D$;
        \item $G$ is regular branch over $\mathrm{St}_G(D-1)$.
    \end{enumerate}
\end{theorem}

By \cite[Theorem 3.7]{JorgeSpectra}, for fractal closed subgroups of $\mathrm{Aut}~T$, the notions of finite type, regular branch and branch are all equivalent. Thus, when proving \textcolor{teal}{Theorems}~\ref{Theorem: f-invariant of fractal of finite type} and \ref{Theorem: classification}, we shall make the, a priori, stronger assumption that the groups under consideration are of finite type.

\subsection{The $f$-invariant of fractal groups of finite type}

We first recall the notion of a cone set. For $n\ge 1$ and $g\in G_n$, we define the \textit{cone set} $C_g$ by 
$$C_g:=\{h\in G\mid h|_\emptyset^n=g\}.$$
In other words, the cone set $C_g$ is simply a coset of $\mathrm{St}_G(n)$.

For a fractal group $G\le \mathrm{Aut}~T$, the $f$-invariant of $(G,\mu_G,T,\mathcal{T})$ is given by $f(T,\alpha_s^n)$ for any $n\ge 1$, where $\alpha_s$ is the \textit{standard partition} of $G$ into the cone sets $\{C_{\sigma}\}_{\sigma\in G_1}$. More generally, the partition $\alpha_s^n$ consists of the cone sets $\{C_g\}_{g\in G_n}$, i.e. the different cosets of $\mathrm{St}_G(n)$ in $G$. Lastly, note that the elements in the partition $\alpha_s^n\vee \mathcal{T}_i^{-1}(\alpha_s^n)$ are sets of the form
\begin{align}
    \label{align: form of the sets in the partition}
    C_g\cap \mathcal{T}_i^{-1}(C_h)=\{k\in G\mid k|_\emptyset^n=g\text{ and }k|_{i}^n=h\},
\end{align}
for $g,h\in G_n$.

We now compute the quantities $F(T,\alpha_s^n)$ for $(G,\mu_G,T,\mathcal{T})$ when $G$ is a fractal group of finite type:

\begin{lemma}
    \label{lemma: The functions F}
    Let $G\le \mathrm{Aut}~T$ be a fractal group of finite type given by patterns of depth~$D$. Let $\alpha_s$ be the standard partition of $(G,\mu_G,T,\mathcal{T})$. Then for any $n\ge D$ we have
    $$F(T,\alpha_s^n)=\log|G_1|-r_{n+1}(G).$$
\end{lemma}
\begin{proof}
    First note that for any $n\ge 1$ we have
    \begin{align}
        \label{align: H as log Gn}
        H(\alpha_s^n)=-\sum_{A\in \alpha_s^n} \mu(A)\log(\mu(A))=\sum_{g\in G_n}\frac{\log|G_n|}{|G_n|}=\log|G_n|.
    \end{align}
    Now let us fix $n\ge D$ for the rest of the proof. Then, since $G$ is a group of finite type given by patterns of depth $D$, we further get by \cref{theorem: equivalence finite type} that $\mathrm{St}_G(n-1)$ is branching and thus
    \begin{align}
    \label{align: reg branch order condition}
        |G_{n+1}|=|G_n|\cdot |\mathrm{St}_G(n):\mathrm{St}_G(n+1)|=|G_n|\cdot |\mathrm{St}_G(n-1):\mathrm{St}_G(n)|^m.
    \end{align}
    Then, for any $1\le i\le m$ and any $A\in \alpha_s^n\vee \mathcal{T}_i^{-1}(\alpha_s^n)$, \cref{align: form of the sets in the partition} together with regular branchness of $G$ over $\mathrm{St}_G(n-1)$ yields
    $$\mu(A)=\frac{|\mathrm{St}_G(n-1):\mathrm{St}_G(n)|^{m-1}}{|G_{n+1}|},$$
    and thus 
    \begin{align*}
        -\log (\mu(A))= \log|G_{n+1}|-(m-1)\log|\mathrm{St}_G(n-1):\mathrm{St}_G(n)|.
    \end{align*}
    Again, \cref{align: form of the sets in the partition} and regular branchness over $\mathrm{St}_G(n-1)$ also yield
    \begin{align*}
        \# (\alpha_s^n\vee \mathcal{T}_i^{-1}(\alpha_s^n))=|G_n|\cdot |\mathrm{St}_G(n-1):\mathrm{St}_G(n)|.
    \end{align*}
    Hence, for any $A\in \alpha_s^n$ and any $1\le i\le m$, we get
    $$\# (\alpha_s^n\vee \mathcal{T}_i^{-1}(\alpha_s^n))\cdot \mu(A)=\frac{|G_n|\cdot|\mathrm{St}_G(n-1):\mathrm{St}_G(n)|^{m}}{|G_{n+1}|}=1.$$
    Therefore, we obtain
    \begin{align*}
        \sum_{i=1}^m H\big(\alpha_s^n\vee \mathcal{T}_i^{-1}(\alpha_s^n)\big)=&-\sum_{i=1}^m\sum_{A\in \alpha_s^n\vee \mathcal{T}_i^{-1}(\alpha_s^n)}\mu(A)\log(\mu(A))\\
        =& m \cdot \# (\alpha_s^n\vee \mathcal{T}_i^{-1}(\alpha_s^n))\cdot\big(-\mu(A)\log (\mu(A))\big) \\
        =&m(-\log(\mu(A)))\\
        =& m \big(\log|G_{n+1}|-(m-1) \log|\mathrm{St}_G(n-1):\mathrm{St}_G(n)|\big).
    \end{align*}
    Putting all the above together and applying \cref{align: reg branch order condition} again, we obtain
    \begin{align*}
        F(T,\alpha_s^n)=&(1-2m)H(\alpha_s^n)+\sum_{i=1}^m H\big(\alpha_s^n\vee \mathcal{T}_i^{-1}(\alpha_s^n)\big)\\
        =&(1-2m)\log|G_n|+m\log|G_{n+1}|-m(m-1)\log|\mathrm{St}_G(n-1):\mathrm{St}_G(n)|\\
        =& \big(\log|G_n|+m\log|\mathrm{St}_G(n-1):\mathrm{St}_G(n)|\big)-m\log|G_n|\\
        &+m \big(-\log|G_n|-m\log|\mathrm{St}_G(n-1):\mathrm{St}_G(n)|+\log|G_{n+1}|\big)\\
        =&\log|G_{n+1}|-m\log|G_n|\\
        =& \log|G_1|-\big(m\log|G_n|-\log|G_{n+1}|+\log|G_1|\big)\\
        =& \log|G_1|-r_{n+1}(G),
    \end{align*}
    by the definition of the sequence $\{r_n(G)\}_{n\ge 1}$.
\end{proof}

\begin{proof}[Proof of \cref{Theorem: f-invariant of fractal of finite type}]
    Note that by \cite[Theorem 3.7]{JorgeSpectra} we may assume $G$ is of finite type of depth $D$ for some $d\ge 1$. Then, the result follows directly from \cref{lemma: The functions F}. Indeed, by \cite[Theorem~3.5]{JorgeSpectra} we have $r_n(G)=r_{D}(G)$ for every $n\ge D$ and thus
    \begin{align*}
        f(G)=&\inf_{n\ge D} F(T,\alpha_s^n)=\inf_{n\ge D} \log|G_1|-r_{n+1}(G)=\log|G_1|-r_{D}(G)=F(T,\alpha_s^D)
    \end{align*}
    as $\alpha_s^D$ is a generating partition. The process $(G,\mu_G,T,\mathcal{T},\alpha_s^D)$ is Markov by \cref{theorem: Markov characterization} as $f(G)=F(T,\alpha_s^D)$.
\end{proof}

\subsection{The Hausdorff dimension of self-similar groups}

As $\mathrm{Aut}~T$ is a profinite group with respect to the level-stabilizer filtration $\{\mathrm{St}(n)\}_{n\ge 1}$, one may define a metric $d:\mathrm{Aut}~T\times \mathrm{Aut}~T\to [0,\infty)$ given by
$$d(g,h)=\inf_{n\ge 1}\{|\mathrm{Aut}~T:\mathrm{St}(n)|^{-1}\mid gh^{-1}\in \mathrm{St}(n)\}$$
for any pair of distinct elements $g,h\in \mathrm{Aut}~T$. This metric induces a Hausdorff dimension on the closed subsets of $\mathrm{Aut}~T$. Given a closed subgroup $G\le \mathrm{Aut}~T$, its Hausdorff dimension in $\mathrm{Aut}~T$ coincides with its lower box dimension \cite{Abercrombie,BarneaShalev}, and it is given by the following lower limit: 
$$\mathrm{hdim}_{\mathrm{Aut}~T}(G)=\underline{\mathrm{dim}}_B(G)=\liminf_{n\to\infty}\frac{\log|G:\mathrm{St}_G(n)|}{\log |\mathrm{Aut}~T:\mathrm{St}(n)|}.$$
In fact, if $G$ is self-similar, then the limit above exists by \cite[Theorem B and Proposition 1.1]{JorgeSpectra}. Therefore, we see that
\begin{align*}
    \mathrm{hdim}_{\mathrm{Aut}~T}(G)&=\lim_{n\to\infty}\frac{\log|G:\mathrm{St}_G(n)|}{\log |\mathrm{Aut}~T:\mathrm{St}(n)|}= \frac{m-1}{\log (m!)}\cdot \lim_{n\to\infty}\frac{\log|G_n|}{m^n-1}\\
    &=C(m)\cdot \lim_{n\to\infty}\frac{\log|G_n|}{m^n}
\end{align*}
for some constant $C(m)$, which only depends on $m$. In other words, the Hausdorff dimension of a self-similar group $G\le \mathrm{Aut}~T$ is completely determined by the limit of the sequence $\{m^{-n}\log|G_n|\}_{n\ge 1}$.

\subsection{A criterion for measure-conjugacy}

If $G\le \mathrm{Aut}~T$ is a group of finite type, we shall write $\mathrm{depth}(G)$ for its depth. Now we proceed with the proof of \cref{Theorem: classification}. The proof is based on the following key observation:

\begin{lemma}
\label{lemma: quotients are of the same size}
    Let $G,H\le \mathrm{Aut}~T$ be two fractal groups of finite type such that:
    \begin{enumerate}[\normalfont(i)]
        \item $f(G)=f(H)$;
        \item $\mathrm{hdim}_{\mathrm{Aut}~T}(G)=\mathrm{hdim}_{\mathrm{Aut}~T}(H)$.
    \end{enumerate}
    Then, for every $n\ge D$ we have
    $$\log|G_n|=\log|H_n|,$$
    where $D:=\max\{\mathrm{depth}(G),\mathrm{depth}(H)\}$.
\end{lemma}
\begin{proof}
    First, note that both $\mathrm{St}_G(D-1)$ and $\mathrm{St}_H(D-1)$ are branching subgroups and thus $r_n(G)=r_D(G)$ and $r_n(H)=r_D(H)$ for every $n\ge D$ again by \cite[Theorem~3.5]{JorgeSpectra}. Then, by \cref{Theorem: f-invariant of fractal of finite type} we get 
    $$\log|G_{D+1}|=\log|G_1|+m\log|G_D|-r_{D+1}(G)=m\log|G_D| + f(G),$$
    and arguing by induction on $k\ge 1$ we obtain
    \begin{align}
        \label{align: log recursive equation}
        \log|G_{D+k}|&=m^k\log|G_D| + \frac{m^k-1}{m-1}\cdot f(G)
    \end{align}
    for any $k\ge 1$. Therefore, if $\log|G_D|=\log|H_D|$ the result follows from \cref{align: log recursive equation}, as we assumed that $f(G)=f(H)$. Furthermore, \cref{align: log recursive equation} also yields the equality $\log|G_D|=\log|H_D|$. Indeed, by \cref{align: log recursive equation}, the Hausdorff dimensions of $G$ and $H$ in $\mathrm{Aut}~T$ are equal if and only if $\log|G_D|=\log|H_D|$, as we have $f(G)=f(H)$ by assumption.
\end{proof}

\begin{proof}[Proof of \cref{Theorem: classification}]
        Let $G,H\le \mathrm{Aut}~T$ be fractal and of finite type and let $D:=\max\{\mathrm{depth}(G),\mathrm{depth}(H)\}$. Then the assumptions of \cref{lemma: quotients are of the same size} are satisfied and we obtain that
         $$\log|G_n|=\log|H_n|$$
        for all $n\ge D$. Furthermore, since $H_n$ and $G_n$ are groups for each $n\ge 1$, the fibers of the projection maps $G_{n+1}\to G_{n}$ are all of the same size for each $n\ge 1$, namely of size $|\mathrm{St}_G(n):\mathrm{St}_G(n+1)|$. Then, as both $\mathrm{St}_G(D-1)$ and $\mathrm{St}_H(D-1)$ are branching, any bijection $f_D:G_D\to H_D$ may be extended for every $n\ge D$ to a bijection $f_n:G_n\to H_n$ in such a way that $f_n(g|_i^{n-1})=f_n(g)|_i^{n-1}$ for every $g\in G_n$ and any $1\le i\le m$.
        
        Since the Haar measure is left-invariant and the quotients $G_n$ and $H_n$ are of the same size for each $n\ge D$, these bijections are measure-preserving and they form a coherence sequence of measure-preserving bijections. Thus, there exists a measure-preserving bijection $f:=\varprojlim f_n:G\to H$ whose inverse is also measure-preserving. By construction
        $$f(\mathcal{T}_v(g))=\mathcal{T}_v f(g)$$
        for any $v\in T$. Furthermore, one has that $f$ induces a bijection between $\alpha_s^D$ and~$\beta_s^D$, where $\alpha_s$ and $\beta_s$ are the standard partitions of $(G,\mu_G)$ and $(H,\mu_H)$ respectively.  Thus, we get that the processes $(G,\mu_G,T,\mathcal{T},\alpha_s^D)$ and $(H,\mu_H,T,\mathcal{T},\beta_s^D)$ are isomorphic.  
\end{proof}

\section{Applications and examples}
\label{section: applications}

We conclude the paper by giving some further applications of the main results in \cref{section: the f-invariant} and working out an example in the $p$-adic tree.

\subsection{Universality of the groups $G_\mathcal{S}$}

Let us assume that $G\le W_q$ for some prime power $q\ge 2$, where
$$W_q:=\{g\in \mathrm{Aut}~T\mid g|_v^1\in \langle \sigma\rangle\le \mathrm{Sym}(q)\}$$
for $\sigma:=(1\dotsb q)\in \mathrm{Sym}(q)$.

 For $G\le \mathrm{Aut}~T$, recall from \cite{JorgeSpectra} (see also \cite{GeneralizedBasilica}) the definition of the sequence $\{s_n(G)\}_{n\ge 1}$:
$$s_n(G):=r_{n+1}(G)-r_n(G)=m \log |\mathrm{St}_G(n-1):\mathrm{St}_G(n)|-\log|\mathrm{St}_G(n):\mathrm{St}_G(n+1)|.$$

Now, note that the same argument as in \cite[Lemma 5.6 and Proposition~5.7(i)]{JorgeSpectra} shows that the sequence $\{s_n(G)\}_{n\ge 1}$ of any self-similar level-transitive group $G$ is an almost $q$-expansion in the sense of \cite{JorgeSpectra}. Thus, for $G\le W_q$ fractal and of finite type, its sequence $\{s_n(G)\}_{n\ge 1}$ is an almost $q$-expansion. Hence, by \cite[Proposition~5.7(i)]{JorgeSpectra}, there exists a super strongly fractal and level-transitive closed subgroup $G_\mathcal{S}\le W_q$ such that $s_n(G_\mathcal{S})=s_n(G)$ for all $n\ge 1$. In particular, by \cite[Theorem~3.5]{JorgeSpectra} and \cref{theorem: equivalence finite type}, the group $G_\mathcal{S}$ is of finite type.

If $G,H\le \mathrm{Aut}~T$ are two fractal groups of finite type such that $s_n(G)=s_n(H)$ for every $n\ge 1$, then $r_n(G)=r_n(H)$ for every $n\ge 1$ too, so 
$$\mathrm{hdim}_{\mathrm{Aut}~T}(G)=\mathrm{hdim}_{\mathrm{Aut}~T}(H)$$
by \cite[Theorem B]{JorgeSpectra}. Then $(G,\mu_G,T,\mathcal{T},\alpha_s^D)$ and $(H,\mu_H,T,\mathcal{T},\beta_s^D)$ are isomorphic by \cref{Theorem: classification}. Thus, by the above discussion, any Markov-process $(G,\mu_G,T,\mathcal{T},\alpha_s^D)$ with $G$ fractal and of finite type is isomorphic to a strongly mixing (in the sense of~\cite{JorgeCyclicity}) Markov-process $(G_\mathcal{S},\mu_{G_\mathcal{S}},T,\mathcal{T},\alpha_s^D)$. 

Note that the above yields countably many non-isomorphic Markov-processes $(G,\mu_g,T,\mathcal{T},\alpha_s^D)$ over each non-abelian free semigroup $T$ of rank a prime power~$q$. Note that there are at most countably many such processes as there are countably many groups of finite type acting on the $q$-adic tree for each prime power $q\ge 2$.

\subsection{Non-fractal groups of finite type}

    Note that fractality of $G$ is not used in the proof of \cref{lemma: The functions F}. The only reason to consider fractal groups is so that we obtain a measure-preserving dynamical system $(G,\mu_G,T,\mathcal{T})$ by \cref{theorem: fractal is dyn sys}, so that we may talk about the associated Markov process. However, we may define the $f$-invariant of a group of finite type in purely group-theoretic terms by \cref{Theorem: f-invariant of fractal of finite type}, i.e. as
    $$f(G):=\log|G_1|-r_{D}(G).$$
    Then, the proof of \cref{lemma: quotients are of the same size} still holds if we drop the fractality condition on $G$ and we obtain the following:

    \begin{corollary}
        Let $G,H\le \mathrm{Aut}~T$ be two groups of finite type. Then, the following are equivalent:
        \begin{enumerate}[\normalfont(i)]
            \item For every $n\ge D$ we get $\log|G_n|=\log|H_n|;$
            \item we have both equalities 
            $$f(G)=f(H) \quad\text{and}\quad \mathrm{hdim}_{\mathrm{Aut}~T}(G)=\mathrm{hdim}_{\mathrm{Aut}~T}(H).$$
        \end{enumerate}
    \end{corollary}

\subsection{An example: GGS-groups}

We first fix some notation. We define the map $\psi:\mathrm{Aut}~T\to (\mathrm{Aut}~T\times\overset{m}{\ldots}\times \mathrm{Aut}~T)\rtimes \mathrm{Sym}(m)$ via
$$g\mapsto (g|_1,\dotsc, g|_m)g|_\emptyset^1.$$
We shall use the map $\psi$ to define automorphisms of $T$ recursively.

Recall also from \cite[Equation (3.1)]{JorgeSpectra} that if a group $G\le \mathrm{Aut}~T$ is regular branch over $\mathrm{St}_G(D-1)$ then
$$r_{D}(G)=\log|G\times\overset{m}{\dotsb}\times G:\psi(\mathrm{St}_{G}(1))|.$$

Let us fix an odd prime $p\ge 3$. Let $\alpha=(\alpha_1,\dotsc,\alpha_{p-1})\in \mathbb{F}_p^{p-1}\setminus\{0\}$ be the so-called \textit{defining vector}. Then, the \textit{GGS-group} $G_\alpha$ is defined as the group $G_\alpha\le W_p$ generated by the rooted automorphism $\psi(a)=(1,\dotsc,1)\sigma$, where $\sigma:=(1\,2\,\dotsb \, p)\in \mathrm{Sym}(p)$, and the directed automorphism $b$ defined recursively as
$$\psi(b)=(a^{\alpha_1},\dotsc, a^{\alpha_{p-1}},b).$$
The group $G_\alpha$ is always strongly fractal. Indeed, note that $G_\alpha$ is level-transitive and that the projections of $b$ and of an appropriate conjugate of $b$ by a power of $a$ at the vertex $p$ generate $G_\alpha$.

If $\alpha$ is not the constant vector, then $G_\alpha$ is branch and thus its closure in $W_p$ is a fractal group of finite type by \cite[Theorem 3.7]{JorgeSpectra}.

The logarithmic orders of the congruence quotients of GGS-groups, and thus the Hausdorff dimensions of their closures in $W_p$, were computed by Fernández-Alcober and Zugadi-Reizabal in \cite{GGSHausdorff}. 
In the proof of \cite[Theorem 3.7]{GGSHausdorff}, the authors proved that if $\alpha$ is not symmetric then
$$r_{D}(G_\alpha)=\log|G_\alpha\times\overset{p}{\dotsb}\times G_\alpha:\psi(\mathrm{St}_{G_\alpha}(1))|=p.$$

Similarly, one can extract from \cite[Theorems 2.1 and 2.14 and Lemma 3.5]{GGSHausdorff} that if $\alpha$ is symmetric but non-constant then
\begin{align*}
    r_{D}(G_\alpha)&=\log|G_\alpha\times\overset{p}{\dotsb}\times G_\alpha:\psi(\mathrm{St}_{G_\alpha}(1))|\\
    &=p\log|G_\alpha:G_\alpha'|+\log|G_\alpha'\times\overset{p}{\ldots}\times G_\alpha':\psi(\mathrm{St}_{G_\alpha}(1)')|-\log|\mathrm{St}_{G_\alpha}(1):\mathrm{St}_{G_\alpha}(1)'|\\
    &=2p+1-p\\
    &=p+1.
\end{align*}
Therefore, \cref{Theorem: f-invariant of fractal of finite type} yields the $f$-invariant of every non-constant GGS-group acting on the $p$-adic tree:
\begin{corollary}
\label{corollary: f invariant GGS}
    Let $\alpha$ be a non-constant defining vector and let us consider the GGS-group $G_\alpha\le W_p$. Then:
    \begin{enumerate}[\normalfont(i)]
        \item if $\alpha$ is not symmetric, we get $f(G_\alpha)=1-p$;
        \item if $\alpha$ is symmetric, we get $f(G_\alpha)=-p$.
    \end{enumerate}
\end{corollary}

In other words, the $f$-invariant of a GGS-group acting on the $p$-adic tree distinguishes precisely whether the non-constant defining vector is symmetric or not. Since this is a measure-conjugacy invariant we see that the Markov-processes associated to a GGS-group given by a non-constant symmetric defining vector cannot be isomorphic to one associated to a non-symmetric defining vector. However, for those symmetric (resp. not symmetric) non-constant defining vectors whose circulant matrix (see \cite{GGSHausdorff}) have the same rank, the corresponding GGS-groups have the same Hausdorff dimension in $W_p$ by \cite[Theorem 3.7]{GGSHausdorff}. Therefore, \cref{Theorem: classification} tells us that, in this case, the associated Markov processes are isomorphic.

% End of document	

% References

\bibliographystyle{unsrt}

\end{document}